\documentclass[12pt,letterpaper]{amsart}

\usepackage{palatino, euler, epic, eepic, xy, microtype, stmaryrd, epsfig, epstopdf}
\usepackage{graphicx,float}

\xyoption{all}

\usepackage{enumerate}
\usepackage[a4paper,left=1in,right=1in,top=1in,bottom=1in,footskip=.25in]{geometry}
\usepackage{caption}
\usepackage{amsmath, amssymb, amsfonts, amsthm, accents, xcolor, enumerate, esint}
\usepackage{mathtools}
\usepackage{tikz-cd}
\usepackage{hyperref}
\usepackage{empheq}
\usepackage{adjustbox}

\usepackage{amsmath, amssymb, amscd, verbatim, xspace,amsthm}
\usepackage{latexsym, epsfig, color}

\newcommand\isom{\cong}

\newcommand\bq{\begin{equation}}
\newcommand\eq{\end{equation}}

\makeatletter
\newenvironment{subtheorem}[1]{%
  \def\subtheoremcounter{#1}%
  \refstepcounter{#1}%
  \protected@edef\theparentnumber{\csname the#1\endcsname}%
  \setcounter{parentnumber}{\value{#1}}%
  \setcounter{#1}{0}%
  \expandafter\def\csname the#1\endcsname{\theparentnumber.\Alph{#1}}%
  \ignorespaces
}{%
  \setcounter{\subtheoremcounter}{\value{parentnumber}}%
  \ignorespacesafterend
}
\makeatother
\newcounter{parentnumber}

\newtheorem{proposition}{Proposition}[section]
\newtheorem{theorem}[proposition]{Theorem}

\newtheorem{example}[proposition]{Example}

\newtheorem{lemma}[proposition]{Lemma}

\theoremstyle{definition}
\newtheorem{definition}[proposition]{Definition}
\newtheorem*{nn-definition}{Definition}

\theoremstyle{remark}
\newtheorem{remark}[proposition]{Remark}
\usepackage{url}

\numberwithin{equation}{section}

\newcommand{\cut}[1]{}

\newcommand\hidden[1]{}

\newcommand{\bfG}{\mathbb{G}}

\newcommand{\PP}{\mathbb{P}}

\newcommand{\ZZ}{{\mathbb{Z}}}                                        %
\newcommand{\cO}{{\mathcal O}}                                        %
\newcommand{\Bl}{\operatorname{Bl}}                                   
\usepackage[makeroom]{cancel}                                                               %

\usepackage{caption}                                                  %
\usepackage{subcaption}                                               %
\usepackage{tikz-cd}                                                  %
\usetikzlibrary{calc}                                                 %
\usetikzlibrary{positioning}                                          %
\usetikzlibrary{decorations.pathreplacing}                            %
\usetikzlibrary{angles}                                               %
\usetikzlibrary{quotes}                                               %
\usetikzlibrary{shapes.misc}                                          %
\usetikzlibrary{decorations.text}                                     %
\usetikzlibrary{arrows}                                               %
\usetikzlibrary{spy}
\usetikzlibrary{intersections}
\usetikzlibrary{through}
\usetikzlibrary{backgrounds}

\usetikzlibrary{arrows.meta}

\usepackage{pgfplots}
\definecolor{pal1}{RGB}{215,48,39}
\definecolor{pal2}{RGB}{252,141,89}
\definecolor{pal3}{RGB}{254,224,144}
\definecolor{pal4}{RGB}{145,191,219}
\definecolor{pal5}{RGB}{69,117,180}
\pgfkeys{/pgf/number format/relative round mode=fixed}
\pgfplotsset{
    contour/label node code/.code={
        \node{\pgfmathprintnumber{#1}\,ms};
    }
}

\title[Curves generating extremal rays in blowups of weighted projective planes]{Curves generating extremal rays in blowups of \\ weighted projective planes}

\author{Javier Gonz\'alez Anaya, Jos\'e Luis Gonz\'alez and Kalle Karu}
\address{
J. Gonz\'alez-Anaya, Department of Mathematics, University of British Columbia, 
  Vancouver, BC V6T1Z2, Canada.  \newline \indent
J.L. Gonz\'alez,  Department of Mathematics, University of California, Riverside,
  Riverside, CA 92521, United States.  \newline \indent
K. Karu,
Department of Mathematics, University of British Columbia, 
  Vancouver, BC V6T1Z2, Canada.} 
\email{jga@math.ubc.ca, jose.gonzalez@ucr.edu, karu@math.ubc.ca}
\thanks{The first author was partially supported by CONACyT scholarship 410172.    \\
\indent The second author was supported by the UCR Academic Senate.    \\
\indent The third author was supported by a NSERC Discovery grant.}

\begin{document}

\begin{abstract}
We consider blowups at a general point of weighted projective planes and, more generally, of toric surfaces with Picard number one. We give a unifying construction of negative curves on these blowups such that all previously known families appear as boundary cases of this. The classification consists of two classes of said curves, each depending on two parameters. Every curve in these two classes is algebraically related to other curves in both classes; this allows us to find their defining equations inductively. For each curve in our classification, we consider a family of blowups in which the curve defines an extremal class in the effective cone. We give a complete classification of these blowups into Mori Dream Spaces and non-Mori Dream Spaces. Our approach greatly simplifies previous proofs, avoiding positive characteristic methods and higher cohomology.
\end{abstract}
\maketitle
\setcounter{tocdepth}{1} 




\section{Introduction}

We work over an algebraically closed field $k$ of characteristic zero.

Let $X$ be the blowup of a weighted projective plane $\PP(a,b,c)$  at a general point $e$. More generally, we allow $X$ to be the blowup of a projective toric surface $X_\Delta$ defined by a triangle $\Delta$:
\[ X= \Bl_e X_\Delta.\]
Such an $X$ is a projective variety of Picard number $2$. Its Mori cone of curves is $2$-dimensional, generated by the class of the exceptional curve $E$ and another class $\gamma$, not necessarily rational, of non-positive self-intersection. If $\gamma$ can be chosen to be the class of an irreducible curve $C$ then we call this uniquely determined curve $C$ a negative curve in $X$. 
It is not known if every $X$ contains a negative curve. As an example, Kurano and Matsuoka \cite{KuranoMatsuoka} conjecture that $X=\Bl_e \PP(9,10,13)$ contains no negative curve. This conjecture, if true, would imply that the Mori cone of $X$ is not rational.  The existence and classification of negative curves is the first main topic of this article.

The existence of a negative curve is closely related to the Mori Dream Space (MDS) property of $X$. 
Recall that a projective variety is called a MDS if its Cox ring is a finitely generated $k$-algebra.
Cutkosky \cite{Cutkosky} shows that $X$ is a MDS if and only if it contains a negative curve $C$ and another curve $D$ disjoint from $C$. There is an extensive literature on proving that certain $X$ is a MDS (see, for example, \cite{Cowsik, Huneke, Cutkosky, Srinivasan}). There are also many proofs of the non-MDS property (for example, \cite{GNW, KuranoNishida, GK, He19, HKL, GGK1, GGK2}).   Our second goal is to unify and simplify many of these results by completely classifying if $X$ in a family is a MDS or a non-MDS.


\vspace{0.5mm}

We let $T\isom \bfG_m^2 \subset X_\Delta$ be the torus, and without loss of generality take $e=(1,1)\in T$. A negative curve in $X$ is defined by a polynomial $f(x,y) \in k[x^{\pm 1}, y^{\pm 1}]$ that vanishes to order $m$ at $e$, and whose Newton polygon lies in a translation and dilation of the triangle $\Delta$ with area $\leq \frac{m^2}{2}$.  By abuse of notation, we will define a negative curve to be a curve in the torus $T$, defined by an irreducible polynomial $f(x,y)$ that vanishes to order $m$ at $e$ and whose Newton polygon fits into {\em some} triangle of area $\leq \frac{m^2}{2}$. Given such a negative curve in the torus, its strict transform defines a negative curve in $X$ for suitable choices of the triangle $\Delta$. 


\vspace{0.5mm}

Having defined negative curves in the torus $T$, we can now try to classify them. For each $m>0$ there is, up to an automorphism of $T$, a finite number of negative curves. Using computer search one can find all these curves when $m$ is small \cite{Wes}. For $m=1$ there is a unique curve defined by $f(x,y)=1-y$. This curve with $m=1$ appeared in all examples of \cite{GNW, GK, He19, HKL}. Similarly, for $m=2$ there is a unique curve defined by $f(x,y)= 3-x-y-x^{-1}y^{-1}$. This case $m=2$ appeared in the examples of \cite{Srinivasan, KuranoNishida}.  For $m=3$ there are two non-isomorphic negative curves and for $m=4$ there are four. In \cite{GGK1, GGK2} we constructed two infinite families of negative curves, giving two non-isomorphic negative curves for each $m\geq 3$. 


\vspace{0.5mm}

All examples of negative curves mentioned above have the property that the Newton polygon of $f(x,y)$ lies in a triangle $\Delta$ that contains exactly ${m+1\choose 2}+1$ lattice points. Since vanishing to order $m$ at $e$ imposes ${m+1\choose 2}$ conditions, the existence of such $f(x,y)$ supported in the triangle is clear; its irreducibility still needs to be proved. There do exist negative curves whose supporting triangle $\Delta$ contains fewer than ${m+1\choose 2}+1$ lattice points. Kurano and Matsuoka \cite{KuranoMatsuoka} gave two such examples. These negative curves seem to be very exceptional and we will not have much to say about them in this article. As an example, if $\Bl_e \PP(9,10,13)$ contains a negative curve, then the curve must be of this exceptional type.


\vspace{0.5mm}

We will only consider negative curves defined by $f(x,y)$ supported in a triangle $\Delta$ with at least ${m+1\choose 2}+1$ lattice points. It  then follows from the uniqueness of the negative curve in $\Bl_e X_\Delta$ that the number of lattice points in $\Delta$ must be exactly ${m+1\choose 2}+1$. This assumption divides the problem of finding negative curves into two steps: first find a triangle of small area and many lattice points, then prove that the polynomial $f(x,y)$ supported in it is irreducible. 


\vspace{0.5mm}

In \cite{GGK1} we gave a criterion for irreducibility of $f(x,y)$ given its Newton polygon. Namely, if a triangle $\Delta$ has an edge whose only lattice points are its endpoints, and the Newton polygon of $f(x,y)$ lies in the triangle and contains the two vertices, then $f(x,y)$ is irreducible. We will impose this irreducibility condition on all negative curves that we study.


\vspace{0.5mm}

We can now state the main theorem in two parts.

\pagebreak

\begin{subtheorem}{proposition}
  \begin{theorem}\label{thm-main-A}
    Let $M,N\geq 0$, $K\geq 3$ be integers satisfying the equation
    \begin{equation} \label{eqn-MN} (M+N)^2 = KMN+1.\end{equation}
    To each such triple we associate an integral triangle $IT(M,N)$ and a rational triangle $RT(M,N)$ with vertices:
    \[ IT(M,N): \quad (0,0), (M+N,KN), (M,0), \]
    \[ RT(M,N): \quad (0,0), (M,M+N), \left(M-\frac{M+N}{K},0\right). \]
    Then, each of these triangles supports a polynomial $f(x,y)$ defining a negative curve that vanishes at $e$ to order $m=M+N$ in the integral case and $m=M$ in the rational case. The negative curves corresponding to $IT(M,N)$ for $M  \geq N > 0$ and $RT(M,N)$ for $M > N > 1$ are pairwise non-isomorphic.
  \end{theorem}
  We omit $K$ from the notation $IT(M,N)$, $RT(M,N)$. When $M$ and $N$ are positive then $K$ is determined by them.  The statement about the curves being non-isomorphic means that there is no automorphism of the torus $T$ that carries one curve to another. An automorphism of the torus is given by an affine linear automorphism of the lattice $\ZZ^2$.
  
In Section~\ref{sec-rec} below we show that the polynomials $f(x,y)$ defining negative curves in  $IT(M,N)$ and $RT(M,N)$ satisfy algebraic relations. Using these relations we can compute the polynomials explicitly.
  
  \begin{theorem}\label{thm-main-B}
    Conversely, if a negative curve is defined by $f(x,y)$ vanishing at $e$ to order $m>0$ and supported in a triangle $\Delta$ that satisfies
    \begin{enumerate}[(a)]
    \item $\Delta$ contains at least  ${m+1\choose 2}+1$ lattice points and has area $\leq \frac{m^2}{2}$,
    \item $\Delta$ has two integral vertices $(0,0)$ and $(m,h)$ where $m$ and $h$ are relatively prime, and a possibly nonintegral vertex $(r,s)$ with $0<r< m$,
    \end{enumerate}
    then the negative curve is isomorphic to one defined in Theorem~\ref{thm-main-A}.
  \end{theorem}
\end{subtheorem}

The theorem classifies all negative curves where the triangles $\Delta$ satisfy conditions (a) and (b). As explained above, condition (a) removes the negative curves of exceptional kind. Condition (b) is used to prove irreducibility of $f(x,y)$. 

All solutions to equation (\ref{eqn-MN}) can be written down explicitly.

\begin{theorem} \phantomsection\label{thm-rec}
Let  $K\geq 3$ be fixed. There are three involutions on the set of integral solutions of  (\ref{eqn-MN}):
\begin{eqnarray*} 
 \iota_0  &: &  (M,N) \mapsto (N,M),\\
  \iota_1  &: &  (M,N) \mapsto (M, (K-2)M-N)\\
 \iota_2  &: &  (M,N) \mapsto ((K-2)N-M,N).
 \end{eqnarray*}
 These involutions satisfy the relation $\iota_2 = \iota_0 \iota_1\iota_0$. See Figure \ref{involutions}. 
 All integral non-negative solutions of equation (\ref{eqn-MN}) can be obtained by starting with $(0,1)$ and applying $\iota_0$ and $\iota_1$ alternately:
 \[ (0,1) \stackrel{\iota_0}{\longmapsto} (1,0) \stackrel{\iota_1}{\longmapsto} (1,K-2) \stackrel{\iota_0}{\longmapsto} (K-2,1)  
\stackrel{\iota_1}{\longmapsto} (K-2, (K-2)^2-1) \stackrel{\iota_0}{\longmapsto} \ldots.\]
\end{theorem}

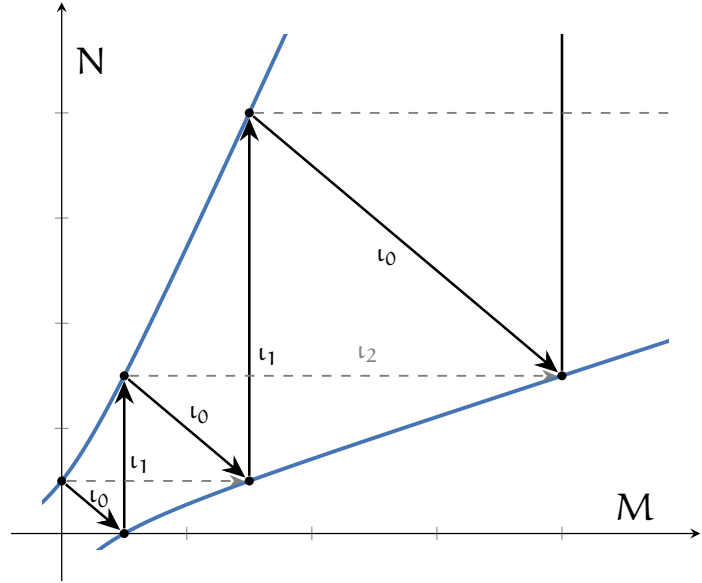
\begin{figure}
  \centering
  \vspace{\baselineskip}
  \scalebox{1.2}{ 
    \begin{tikzpicture}[
  trim axis left,
  trim axis right
  ]
  \begin{axis}[
    title={},
    xmin=-0.3,
    ymin=-0.3,
    xmax=9.7,
    ymax=9.5,
    axis x line=center, axis y line=center,
    enlargelimits=false,
    minor tick num=0,
    axis line style={shorten >=-10pt, shorten <=-10pt},
    xlabel=$M$,
    ylabel=$N$,
    xticklabels=\empty,
    yticklabels=\empty,
    ]
    \addplot +[pal5,no markers,
      raw gnuplot,
      very thick,
      empty line = jump 
      ] gnuplot {
      set contour base;
      set cntrparam levels discrete 0.0003;
      unset surface;
      set view map;
      set isosamples 200;
      set samples 200;
      splot (x+y)^2-5*x*y-1;
    };
    \node[label={180:{\tiny $1$}},circle,fill,inner sep=1pt] at (axis cs:0,1) (A) {};
    \node[label={135:{\tiny}},circle,fill,inner sep=1pt] at (axis cs:1,3) (C) {};
    \node[label={135:{\tiny}},circle,fill,inner sep=1pt] at (axis cs:3,8) (E) {};
    \node[label={135:{\tiny}},circle,fill,inner sep=1pt] at (axis cs:8,12) (G) {};
    \node[label={270:{\tiny $1$}},circle,fill,inner sep=1pt] at (axis cs:1,0) (B) {};
    \node[label={315:{\tiny}},circle,fill,inner sep=1pt] at (axis cs:3,1) (D) {};
    \node[label={315:{\tiny}},circle,fill,inner sep=1pt] at (axis cs:8,3) (F) {};

    \node[label={45:{\tiny $\iota_0$}}] at (axis cs:0,0) () {};
    \node[label={45:{\tiny $\iota_1$}}] at (axis cs:0.65,0.65) () {};
    \node[label={45:{\tiny $\iota_0$}}] at (axis cs:1.6,1.6) () {};
    \node[label={45:{\tiny $\iota_1$}}] at (axis cs:2.7,2.7) () {};
    \node[label={45:{\tiny $\iota_0$}}] at (axis cs:4.6,4.6) () {};
    \draw [-{Stealth[scale=1.1,angle'=45]},thick] (A) -- (B);
    \draw [-{Stealth[scale=1.1,angle'=45]},thick] (B) -- (C) node[midway,right] {\tiny};
    \draw [-{Stealth[scale=1.1,angle'=45]},thick] (C) -- (D) node[midway,above] {\tiny};
    \draw [-{Stealth[scale=1.1,angle'=45]},thick] (D) -- (E) node[midway,right] {\tiny};
    \draw [-{Stealth[scale=1.1,angle'=45]},thick] (E) -- (F) node[midway,above] {\tiny};
    \draw [-{Stealth[scale=1.1,angle'=45]},thick] (F) -- (G) node[midway,right] {\tiny};

    \draw [-{Stealth[gray,scale=0.8,angle'=45]},gray,semithick,dashed] (A) -- (D) node[midway,above right] {\tiny};
    \draw [-{Stealth[gray,scale=0.8,angle'=45]},gray,semithick,dashed] (C) -- (F) node[midway,above right] {\tiny $\iota_2$};
    \draw [-{Stealth[gray,scale=0.8,angle'=45]},gray,semithick,dashed] (E) -- (axis cs:12,8) node[midway,above right] {\tiny};
  \end{axis}
\end{tikzpicture}

  }
  \caption{Involutions generating all integral solutions to  $(M+N)^2=KMN+1$ as described in Theorem \ref{thm-rec}}
  \label{involutions}
\end{figure}

One can see from the previous theorem that for $K=3$ there are three solutions $(0,1), (1,0), (1,1)$.  For $K\geq 4$ there is an infinite sequence of solutions. The integral and rational cases then give us two $2$-parameter families of distinct negative curves. 

For $K\geq 4$, the solutions $(M,N)$ can be divided into two sets, $M>N$ and $M<N$. The triangles in the two sets are isomorphic.  More precisely, there exist affine linear automorphisms of the lattice $\ZZ^2$ that induce isomorphisms of triangles:
\[ IT(M,N) \isom IT(\iota_0(M,N)),\]
\[ RT(M,N) \isom RT (\iota_1(M,N)).\]

\begin{example}\label{ex-families}
\begin{enumerate}
\item The triples $M\geq N=1$, $K=M+2$ are  solutions of (\ref{eqn-MN}). The corresponding rational triangles have vertices $(0,0), (M, M+1), (M-1+1/K,0)$. This is the family of negative curves studied in \cite{GGK1}.
\item When $K=4$ then $M\geq 2, N=M-1$ are solutions of  (\ref{eqn-MN}). The corresponding rational triangles  have vertices $(0,0), (M, 2M-1), \left(M-\frac{2M-1}{4}, 0\right)$. This is the family of negative curves studied in \cite{GGK2}.
\item The integral triangle corresponding to $K=4, M=3, N=2$ is the simplest one that does not appear in either of the two families in \cite{GGK1, GGK2}. It has vertices $(0,0), (5,8), (0,3)$ and $m=5$.
\item Let us look for negative curves with $m=4$. There is one integral triangle $IT(3,1)$, $K=5$, and one rational triangle $RT(4,3)$, $K=4$. There are two additional negative curves with $m=4$, see Figure \ref{m4}.
\end{enumerate}
\begin{figure}
  \centering
  \vspace{0.4cm}
  \begin{subfigure}[b]{0.4\textwidth}
  \centering
  \scalebox{0.8}{
  \begin{tikzpicture}[trim axis left,
    trim axis right
    ]
    \begin{axis}[
    title={},
    xmin=-0.2,
    ymin=-0.2,
    xmax=4.1,
    ymax=5.1,
    axis x line=center, axis y line=center,
    enlargelimits=true,
    minor tick num=0,
    axis line style={shorten >=-10pt, shorten <=-10pt},
    xlabel=$x$,
    ylabel=$y$,
    xticklabels=\empty,
    yticklabels=\empty,
    ]
    
    \foreach \x/\y in {0/0,1/0,2/0,3/0,1/1,2/1,2/2,3/1,3/2,3/3,4/5}
    {
      \edef\temp{\noexpand\draw[fill] (axis cs:\x,\y) circle (1.6pt);}
      \temp
    }
    \draw (axis cs:0,0) -- (axis cs:4,5) -- (axis cs: 3,0) -- cycle;
    \node[label={90:{\tiny $(4,5)$}},circle,fill,inner sep=1pt] at (axis cs:4,5) () {};
  \end{axis}
\end{tikzpicture}
}
\vspace{0.2cm}
\caption[Ex2]%
{$IT(3,1)$ with $K=5$, member of the family from \cite{GGK1} with $C\cdot C=-1$.}    
\end{subfigure}
\vspace{0.5cm} 
\hspace{1cm}
\begin{subfigure}[b]{0.4\textwidth}  
  \centering
  \scalebox{0.8}{
    
  \begin{tikzpicture}[
  trim axis left,
  trim axis right
  ]
  \begin{axis}[
    title={},
    xmin=-0.2,
    ymin=-0.2,
    xmax=4.1,
    ymax=7.1,
    axis x line=center, axis y line=center,
    enlargelimits=true,
    minor tick num=0,
    axis line style={shorten >=-10pt, shorten <=-10pt},
    xlabel=$x$,
    ylabel=$y$,
    xticklabels=\empty,
    yticklabels=\empty,
    ]
    
    \foreach \x/\y in {0/0,1/0,2/0,1/1,2/1,2/2,2/3,3/3,3/4,3/5,4/7}
    {
      \edef\temp{\noexpand\draw[fill] (axis cs:\x,\y) circle (1.6pt);}
      \temp
    }
    \draw (axis cs:0,0) -- (axis cs:4,7) -- (axis cs: 2.25,0) -- cycle;
    \node[label={90:{\tiny $(4,7)$}},circle,fill,inner sep=1pt] at (axis cs:4,7) () {};
    \node[label={270:{\tiny $(9/4,0)$}}] at (axis cs:9/4,0) () {};
  \end{axis}
\end{tikzpicture}
}
\vspace{0.2cm}
  \caption[]%
{$RT(4,3)$ with $K=4$, member of the family from \cite{GGK2} with $C\cdot C=-\frac{1}{4}$.}    
\end{subfigure}
\begin{subfigure}[b]{0.4\textwidth}   
  \centering
  \scalebox{0.8}{

  \begin{tikzpicture}[
    trim axis left,
    trim axis right
    ]
    \begin{axis}[
      title={},
      xmin=-0.2,
      ymin=-0.2,
      xmax=5.1,
      ymax=7.2,
      axis x line=center, axis y line=center,
      enlargelimits=true,
      minor tick num=0,
      axis line style={shorten >=-10pt, shorten <=-10pt},
      xlabel=$x$,
      ylabel=$y$,
      xticklabels=\empty,
      yticklabels=\empty,
    ]
    
    \foreach \x/\y in {0/0,1/0,2/0,1/1,2/1,2/2,3/2,3/3,3/4,4/5,5/7}
    {
      \edef\temp{\noexpand\draw[fill] (axis cs:\x,\y) circle (1.6pt);}
      \temp
    }
    \draw (axis cs:0,0) -- (axis cs:5,7) -- (axis cs: 2.2,0) -- cycle;
    \node[label={90:{\tiny $(5,7)$}}] at (axis cs:5,7) () {};
    \node[label={270:{\tiny $(11/5,0)$}}] at (axis cs:2.2,0) () {};
  \end{axis}
  \end{tikzpicture}
}
\vspace{0.2cm}
  \caption[]%
  {Negative curve in $\operatorname{Bl}_e\mathbb{P}(5,7,11)$ with $C\cdot C=-\frac{3}{5}$.}
\end{subfigure}
\hspace{1cm}
\begin{subfigure}[b]{0.4\textwidth}   
  \centering
  \scalebox{0.8}{
  \begin{tikzpicture}[
    trim axis left,
    trim axis right
    ]
    \begin{axis}[
    title={},
    xmin=-0.2,
    ymin=-0.2,
    xmax=4.6,
    ymax=14.2,
    axis x line=center, axis y line=center,
    enlargelimits=true,
    minor tick num=0,
    axis line style={shorten >=-10pt, shorten <=-10pt},
    xlabel=$x$,
    ylabel=$y$,
    xticklabels=\empty,
    yticklabels=\empty,
    ]
    
    \foreach \x/\y in {0/0,1/0,1/1,1/2,1/3,2/4,2/5,2/6,3/9,3/10,2/4,4/13}
    {
      \edef\temp{\noexpand\draw[fill] (axis cs:\x,\y) circle (1.6pt);}
      \temp
    }
    \draw (axis cs:0,0) -- (axis cs:30/7,100/7) -- (axis cs: 10/9,0) -- cycle;
    \node[label={270:{\tiny $(10/9,0)$}}] at (axis cs:10/9,0) () {};
    \node[label={90:{\tiny $\frac{10}{7}(3,10)$}}] at (axis cs:30/7,13.6) () {};
    \node[label={0:{\tiny $(4,13)$}}] at (axis cs:4,13) () {};
  \end{axis}
\end{tikzpicture}
}
\vspace{0.2cm}
\caption[]%
{Negative curve in $\operatorname{Bl}_e\mathbb{P}(7,9,10)$ with $C\cdot C=-\frac{8}{63}$.}    
\end{subfigure}


  \caption{Every negative curve with $m=4$.}
  \label{m4}
\end{figure}
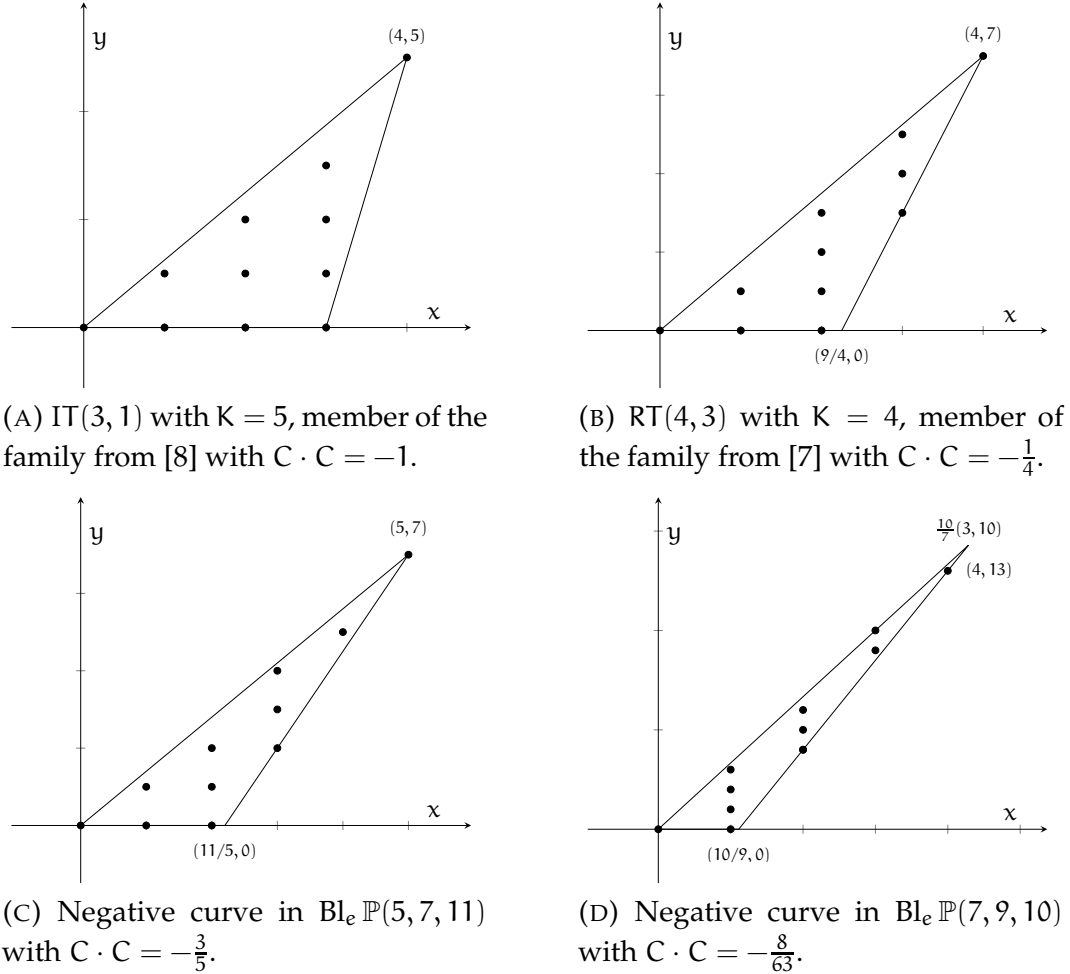
\end{example}

In \cite{GGK1, GGK2} we studied two families of negative curves (Example~\ref{ex-families} (1), (2) ) and for each such curve we constructed examples of MDS and non-MDS. We generalize these results to all negative curves described above.

Let $\Delta$ be one of the triangles $IT(M,N)$ or $RT(M,N)$ described in Theorem~\ref{thm-main-A}, and let $C^0\subseteq T$ be the corresponding negative curve. We consider larger triangles $\Delta_{\alpha,\beta}$ obtained from $\Delta$ by increasing its base:
  \[ IT(M,N)_{\alpha,\beta}: \quad (-\alpha,0), (M+N,KN), (M+\beta,0), \]
  \[ RT(M,N)_{\alpha,\beta}: \quad (-\alpha,0), (M,M+N), \left(M-\frac{M+N}{K}+\beta,0\right). \]
Here $\alpha,\beta\geq 0$. Let $X= \Bl_e X_{\Delta_{\alpha,\beta}}$. Then, the strict transform of $C^0$ is a negative curve in $X$ if $Area(\Delta_{\alpha,\beta}) \leq \frac{m^2}{2}$. This is equivalent to $\alpha+\beta \leq 1/NK$ in the integral case and $\alpha+\beta \leq 1/K(M+N)$ in the rational case.

\begin{theorem}\label{thm-MDS}
Let the variety $X$ be constructed from either $IT(M,N)$ or $RT(M,N)$ by choosing $\alpha,\beta\geq 0$ such that the strict transform of $C^0$ is a negative curve in $X$.
\begin{enumerate}
\item If $\alpha=0$ or $\beta=0$ then $X$ is a MDS.
\item If $\alpha>0$ and $\beta>0$ then $X$ is a non-MDS when $N>1$ in the integral case and $M+N>1$ in the rational case.
\end{enumerate}
\end{theorem}

\begin{remark}
The theorem applies with $M>N$ and with $M<N$. Even though these triangles are isomorphic, when choosing $\alpha,\beta\geq 0$ we lengthen different edges of the same triangle. 
\end{remark}

\begin{remark}
The negative curve $C$ is allowed to have self-intersection number zero. The theorem gives many examples of non-MDS where $C$ has zero self-intersection. 
\end{remark}

\begin{remark}
  The triangles $IT(K-2,1)$ and $RT(K-1,1)$ support the same negative curve: the former is the convex hull of the lattice points in the latter. When choosing a set of non-isomorphic curves in Theorem \ref{thm-main-A} we have discarded the triangles $RT(K-1,1)$. In contrast,  the second part of Theorem \ref{thm-MDS} applies to $RT(K-1,1)$, but not to $IT(K-2,1)$.
 These negative curves are studied in more detail in \cite{GGK1}. 
  \end{remark}

\begin{remark}
For small values of $M$ and $N$, the possibly degenerate triangles $IT(0,1)$, $IT(1,0)$, $RT(1,0)$, $RT(1,1)$  support the same  negative curve with vanishing order $m=1$. This case has been widely studied in the literature \cite{GNW,GK,KuranoNishida,HKL,He19}. The present paper does not say anything more about the $m=1$ case. 
\end{remark}

\section{Existence of negative curves}\label{dagger}

Recall that a negative curve in $X=\Bl_e X_\Delta$ is an irreducible curve $C$ of non-positive self-intersection, different from the exceptional curve $E$. If the curve $C$ has strictly negative self-intersection then it is unique in $X$.

Let us say that a triangle $\Delta$ supports a negative curve if $\Delta$ has area $\leq \frac{m^2}{2}$ and there exists a polynomial $f(x,y)\in k[x^{\pm 1}, y^{\pm 1}]$ with Newton polygon in $\Delta$ that vanishes to order $m$ at the point $e=(1,1)$. Such an $f(x,y)$ defines a curve $C$ in $X=\Bl_e X_\Delta$ with self-intersection number
\[ C\cdot C = 2 Area(\Delta) - m^2 \leq 0.\]
If $\Delta$ supports a negative curve with strictly negative self-intersection then the polynomial $f(x,y)$ is unique up to a constant multiple. This implies that $\Delta$ can contain at most ${ m+1 \choose 2} +1$ lattice points because vanishing at $e$ imposes ${ m+1 \choose 2}$ conditions. 

Let us define a form of triangles that will appear in the proofs below. The triangles $IT(M,N)$ and $RT(M,N)$ are of this form.

\begin{definition} We say that a triangle $\Delta$ is of the form $(\dagger)$ if it has vertices
\[ (0,0), (b,0)\text{ and } (m,h), \]
where $m, h >0 $ are relatively prime integers, $0<b < m$ is a rational number, and the slope of the right edge
\[ K = \frac{h}{m-b} \] 
is an integer. 
\end{definition}

The main goal of this section is to give a sufficient condition for such a triangle $\Delta$ to support a negative curve.

\begin{proposition}\label{prop-neg-curve}
Let $\Delta$ be a triangle of the form $(\dagger)$ with area $\leq \frac{m^2}{2}$ and containing at least ${m+1 \choose 2}+1 $ lattice points. Then $\Delta$ supports an irreducible negative curve vanishing to order $m$ at $e$.
\end{proposition}

The triangle $\Delta$ clearly supports a polynomial $f(x,y)$ that vanishes to  order $m$ at $e$. We need to prove that this $f(x,y)$ is irreducible. Let us recall an irreducibility criterion proved in \cite{GGK1}.

\begin{lemma} \label{lem-irr}
Let $\Delta$ be a triangle and $f(x,y)$ a polynomial supported in $\Delta$. Suppose an edge of $\Delta$ intersects $\ZZ^2$ at its  endpoints only and these endpoints lie in the support of $f(x,y)$. Then $f(x,y)$ is irreducible.
\end{lemma}

The lemma is proved by showing that the Newton polygon of $f(x,y)$ cannot be written as the Minkowski sum of two smaller polygons.

In the case of the triangle $\Delta$ of the form $(\dagger)$, we wish to prove that the two vertices $(0,0)$ and $(m,h)$ lie in the Newton polygon of $f(x,y)$. We will see below that the triangle $\Delta$ contains exactly ${m+1 \choose 2}+1 $ lattice points. Thus, if one of the vertices does not lie in the Newton polygon of $f(x,y)$, then we are in the exceptional situation  where vanishing at $e$ to order $m$ imposes ${m+1 \choose 2}$ conditions on the same number of monomials. It follows that one of these conditions must be trivial. This can be stated in terms of lattice point interpolation:

\begin{lemma}\label{interpolation-lemma}
Let $S$ be a set of ${m+1 \choose 2}$ lattice points on the plane.
Then, $S$ supports a Laurent polynomial vanishing to order $m$ at $e=(1,1)$ if and only if there is a degree $m-1$ curve interpolating all points in $S$. 
\end{lemma} 

\begin{proof} 
A polynomial $f(x,y)$ supported on $S$ vanishes to order $m$ at $e$ if all partial derivatives  up to order $m-1$ vanish when applied to $f(x,y)$ and evaluated at $e$. The same is true if we replace partial derivatives with logarithmic partial derivatives $p(x\partial_x, y\partial_y)\in k[x\partial_x, y\partial_y]$. When the number of monomials is the same as the number of conditions given by derivatives, then a nontrivial solution $f(x,y)$ exists if and only if one condition is trivial, meaning some logarithmic partial derivative $p$ vanishes on all monomials in $S$ when evaluated at $e$. Now
\[ p(x\partial_x, y\partial_y) (x^a y^b)|_{(x,y)= (1,1)} = p(a,b).\]
This $p$ is a polynomial of degree at most $m-1$ that vanishes at all lattice points in $S$.
\end{proof}

The previous lemma is well-known. See for example \cite{Dumnicki06, Castravet1,He19}. Let us use it to prove Proposition~\ref{prop-neg-curve} with some additional assumptions. 

\begin{lemma} \label{lem-irred-curve}
Let $\Delta$ be a triangle of the form $(\dagger)$ with area $\leq \frac{m^2}{2}$ such that the number of lattice points in the $m+1$ columns of $\Delta$ is $1,1,2,3,\ldots, m$ (possibly permuted). Then $\Delta$ supports an irreducible negative curve vanishing to order $m$ at $e$.
\end{lemma}

\begin{proof}
Notice that the number of lattice points in $\Delta$ is $1+1+2+\ldots+m =  {m+1 \choose 2}+1$. This implies that there exists a polynomial $f(x,y)$ supported in $\Delta$ and vanishing to order $m$ at $e$. If the Newton polygon of $f(x,y)$ does not contain one of the vertices $(0,0)$ or $(m,h)$, then there must be a degree $m-1$ curve through the remaining lattice points that lie in columns of size $1,2,3,\ldots, m$. This is not possible by Bezout's theorem: the curve must consist of vertical lines along the columns with $m, m-1,\ldots,2$ points, after which there is still one point left over.
\end{proof}

It remains to prove that a triangle as in Proposition~\ref{prop-neg-curve} indeed contains the correct number of lattice points in its columns.

Consider $\Delta$ of the form $(\dagger)$. Let us divide the triangle $\Delta$ into two smaller triangles using the vertical line $x=b$. The two smaller triangles have a new common vertex 
\[ \left(b, b \frac{h}{m}\right).\]
To the right hand triangle we apply the shear transformation
\[ (x,y) \mapsto (x, y-K(x-b)).\]
Let $\Delta^{sh}$ be the union of the left triangle and the sheared right triangle. Then $\Delta^{sh}$ is again a triangle with vertices
\[ (0,0), (m,0)\text{  and  } \left(b, b \frac{h}{m}\right).\]
The shear transformation maps lattice points to lattice points and preserves columns. Hence, $\Delta$ and $\Delta^{sh}$ contain the same number of lattice points column by column. See Figure \ref{shear}. 

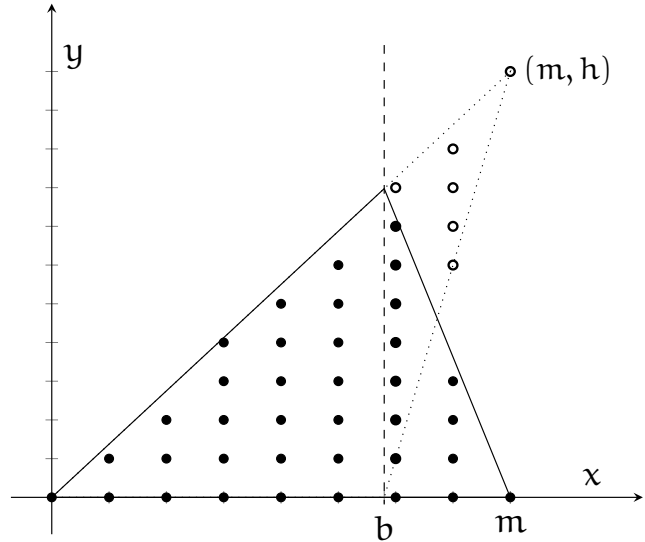
\begin{figure}[H]
  \centering
  \scalebox{1.1}{
    \begin{tikzpicture}[trim axis left, trim axis right]
  \begin{axis}[
    title={},
    xmin=-0.2,
    ymin=-0.2,
    xmax=9.8,
    ymax=12,
    axis x line=center, axis y line=center,
    enlargelimits=false,
    minor tick num=0,
    axis line style={shorten >=-10pt, shorten <=-10pt},
    xticklabels={,,},
    yticklabels={,,},
    xtick={1,2,...,8},
    ytick={1,2,...,11},
    xlabel=$x$,
    ylabel=$y$,
    extra x ticks={5.8,8},
    extra x tick labels={$b$,$m$}
    ]
    \node[label={0:{\small $(m,h)$}}] at (axis cs: 7.8,11) {};
    \node at (axis cs:5.8,12) (A) {};
    \node at (axis cs:5.8,-0.5) (B) {};
    \draw [dashed] (A) -- (B);
    \draw [dotted] (axis cs:0,0) -- (axis cs:8,11) -- (axis cs: 29/5,0) -- cycle;
    \foreach \x/\y in {6/1,6/2,6/3,6/4,6/5,6/6,6/7,6/8,7/6,7/7,7/8,7/9,8/11}
    {
      \edef\temp{\noexpand\draw[thick] (axis cs:\x,\y) circle (1.5pt);}
      \temp
    }
    \foreach \x/\y in {0/0,1/0,1/1,2/0,2/1,2/2,3/0,3/1,3/2,3/3,3/4,4/0,4/1,4/2,4/3,4/4,4/5,5/0,5/1,5/2,5/3,5/4,5/5,5/6,6/0,6/1,6/2,6/3,6/4,6/5,6/6,6/7,7/0,7/1,7/2,7/3,8/0}
    {
      \edef\temp{\noexpand\draw[fill] (axis cs:\x,\y) circle (1.5pt);}
      \temp
    }
    \draw (axis cs:0,0) -- (axis cs:8,0) -- (axis cs: 29/5,319/40) -- cycle;
  \end{axis}
\end{tikzpicture}
    }
    \caption{Comparison between $\Delta$ and $\Delta^{sh}$ for $RT(8,3)$ with $K=5$.}
    \label{shear}
\end{figure}

\begin{lemma} \label{column.point.count} 
Let $\tilde{\Delta}$ be the triangle with vertices $(0,0), (m,0), (b', m)$, where $m > 0$ is an integer and $0< b' < m$ is rational. If the only lattice points on the boundary of $\tilde{\Delta}$ are on its lower edge,  
then the number of lattice points in the columns of $\tilde{\Delta}$ is $1,1,2,3,\ldots, m$ (possibly permuted).
\end{lemma}

\begin{proof}
By assumption, $\tilde{\Delta}$ has both width and height equal to $m$. Hence the slice of the triangle with $y$ constant is an interval of length $m-y$.  Since the only lattice points on the boundary of the triangle are the points on its lower edge, then the number of lattice points in the rows $y=0, 1,  \ldots, m-1$ is $m+1, m-1, m-2, \ldots, 1$. 

Consider a column of lattice points in $\tilde{\Delta}$, $(c,0), (c,1),\ldots, (c,d)$. Let $(c, d+1)$ be the lattice point above the column; we count these points. When we add a row with $r$ lattice points on top of a row with $r+1$ lattice points, we create one such lattice point above a column. This gives the stated number of lattice points in the columns of $\tilde{\Delta}$.  
\end{proof}

\begin{proof}[Proof of Proposition~\ref{prop-neg-curve}]
If the area of $\Delta$ is $\frac{m^2}{2}$, then $b \frac{h}{m} = m$ and the conclusion follows by Lemma \ref{column.point.count} applied to $\Delta^{sh}$. 
To deal with the case where the area is less than $\frac{m^2}{2}$, we compare $\Delta^{sh}$ with a larger triangle $\tilde{\Delta}^{sh}$ that has area $\frac{m^2}{2}$.
Let  $\tilde{\Delta}^{sh}$ have vertices 
\[ (0,0), (m,0), (b', m)\]
for some rational $0 < b' < m$ such that  $\Delta^{sh} \subseteq \tilde{\Delta}^{sh}$. We can choose $b'$ so that  $\tilde{\Delta}^{sh}$ has no lattice points on its boundary except on the lower edge. By Lemma \ref{column.point.count}, $\tilde{\Delta}^{sh}$ has columns $1, 1, 2, \ldots, m$. If we now know that $\Delta$ (and hence also $\Delta^{sh}$) contains at least ${m+1 \choose 2}+1$ lattice points, then 
\[ \Delta^{sh}\cap \ZZ^2 = \tilde{\Delta}^{sh} \cap \ZZ^2\]
and hence both $\Delta^{sh}$ and $\Delta$ must also have columns $1,1,2,\ldots,m$. The claim now follows from Lemma~\ref{lem-irred-curve}.
\end{proof}

\section{Proofs of Theorem~\ref{thm-main-A} and Theorem~\ref{thm-rec}}    \label{section.proofs.Thm1.1.A.Thm1.2}

We start by proving that the triangles $IT(M,N)$ and $RT(M,N)$ support a negative curve. This follows from Proposition~\ref{prop-neg-curve} once we know that the triangles contain enough lattice points because the triangles are of the form $(\dagger )$ defined at the beginning of Section \ref{dagger}. 

\begin{lemma}     \label{lemma.point.count.IT.RT}
The triangles $IT(M,N)$ and $RT(M,N)$ contain ${m+1 \choose 2}+1 $ lattice points, where $m=M+N$ in the integral case and $m=M$ in the rational case.
\end{lemma}

\begin{proof}
We apply the Pick's Theorem to the integral triangle $\Delta=IT(M,N)$:
\[ |\Delta \cap \ZZ^2| = Area(\Delta) +\frac{P(\Delta)}{2} +1,\]
where $P(\Delta)$ is the perimeter of $\Delta$. The triangle has base $M$ and height $KN$, hence
\[2Area(\Delta) = KMN = (M+N)^2-1 = m^2-1.\]
The boundary of the triangle contains exactly one lattice point for every integral value of $x$ between $0$ and $M+N$, hence $P(\Delta)= M+N+1=m+1$. This gives  
\[ |\Delta \cap \ZZ^2| = \frac{m^2-1}{2} + \frac{m+1}{2} +1 = {m+1 \choose 2} +1.\]

Let now $\Delta=RT(M,N)$. We will apply Pick's Theorem to the convex hull of the lattice points in $\Delta$. Let us find the lowest lattice point on the right edge of the triangle. From equation (\ref{eqn-MN}) it follows that $K$ divides $M+N \pm 1$. This implies that the lowest lattice point on the right edge of $\Delta$ has $y$-coordinate $1$ or $K-1$. In either case the convex hull of lattice points in $\Delta$ is a $4$-gon with area 
\[ Area(\Delta)-\frac{1}{2}\left(1-\frac{1}{K}\right) = \frac{m^2-1}{2}. \] 
Again, the boundary of $\Delta$ contains one lattice point for every integral value of $x$ from $0$ to $M$. Pick's Theorem now says that $\Delta$ contains ${m+1 \choose 2} +1$ lattice points.
\end{proof}

This proves the first part of Theorem~\ref{thm-main-A}. 

Next we study isomorphisms between these negative curves. Recall that two such curves are isomorphic if there exists an automorphism of the torus $T$ mapping one curve to another. An  automorphism of $T$ is given by an integral affine linear automorphism of $\ZZ^2$. Two negative curves are then isomorphic only if such an integral affine linear automorphism maps the Newton polygon of one $f(x,y)$ to the Newton polygon of another. 

\begin{lemma} 
The  only integral affine linear isomorphisms between the triangles $IT(M,N)$ and $RT(M,N)$ are
\[ IT(M,N) \isom IT(\iota_0(M,N)),\]
\[ RT(M,N) \isom RT (\iota_1(M,N)).\]
\end{lemma}

\begin{proof}
The first isomorphism is given by the transformation
\[ (x,y) \mapsto (M+N-x, K(M-x)+y),\]
the second one by 
\[ (x,y) \mapsto (M-x, K(M-x)+y-(M+N)).\]

These two transformations map the left edge of one triangle to the left edge of the other and exchange the other two edges.

To prove that there are no other isomorphisms we consider the normal fan of the triangle. It has three maximal cones with integer multiplicities $m_1, m_2, m_3$. This set of multiplicities is preserved under an isomorphism. The multiplicities are $NK, K$ and $MK$ in the integral case and $M+N, K$ and $MK-M-N$ in the rational case.

Clearly no integral triangle is isomorphic to a rational triangle. In the integral case the set $\{NK, K, MK\}$ determines $K$ (the GCD of the triple) and the set $\{N, M\}$. Thus, only the triangles that differ by $\iota_0$ have the same triple of multiplicities.

In the rational case, the multiplicity $K$ corresponds to the unique non-integral vertex. If $K$ is known then from the set $\{ M+N, MK-(M+N)\}$ we recover $M$. The only triangles with the same $K$ and $M$ are the two that differ by $\iota_1$.
\end{proof}

Let us now prove the second part of Theorem~\ref{thm-main-A}, showing that the negative curves are pairwise non-isomorphic. 
We will restrict to the curves in the statement of the theorem, namely those corresponding to $IT(M,N)$ for $M\geq N>0$ and those corresponding to $RT(M,N)$ for $M>N>1$. We will also restrict to $K\geq 4$, the case $K=3, M=N=1$ giving the unique negative curve that vanishes to order $2$ at $e$.

To show that these curves are non-isomorphic it suffices to prove that the Newton polygons of the defining polynomials $f(x,y)$ are non-isomorphic. We show that the Newton polygon determines the triangle uniquely, hence by the previous lemma, no two Newton polygons can be isomorphic.

In Lemma \ref{lem-IT} below we show that in the integral case the Newton polygon coincides with the triangle. This shows that no two negative curves supported in the integral triangles are isomorphic to each other or to a negative curve supported in a rational triangle.

For rational triangles the Newton polygons are harder to determine. In general they are not equal to the convex hull of lattice points in the triangle.  What follows is a proof that the Newton polygon nevertheless uniquely determines the triangle.

In Lemma~\ref{lem-Newt} below we show that in addition to the left edge of the triangle, the Newton polygon also contains edges along the other two sides of the triangle. Suppose now that the Newton polygon uniquely determines one of its edges, the one coinciding with the left edge of the triangle. Then, it determines the triangle by  extending the two adjacent edges until they meet.

Consider the edge in the Newton polygon that coincides with the left edge of the triangle. We will write down properties of this edge that uniquely determine it among all edges of the Newton polygon. Consider the following properties for an edge in the Newton polygon of $f(x,y)$:
\begin{enumerate}
\item There exists a triangle that contains the Newton polygon and shares the same edge.
\item The lattice length of the edge is $1$.
\item The coefficients of the two monomials in $f(x,y)$ corresponding to the vertices of the edge have the same absolute value. 
\item The triangle constructed from the edge by extending the two adjacent sides must contain ${m+1 \choose 2} + 1$ lattice points for some $m$.
\end{enumerate}

The Newton polygon corresponding to $RT(M,N)$ satisfies (1) - (4) with respect to the edge from $(0,0)$ to $(M,M+N)$. Indeed, (1) and (2) are clear, (3) is shown in Section~\ref{sec-rec}, and (4) is shown in Lemma~\ref{lemma.point.count.IT.RT}. 
In an arbitrary convex polygon that is not a triangle there can be at most two edges satisfying property (1), and these edges have to be adjacent. Aside from the left hand side edge, the only other edge of the Newton polygon of the negative curve of $RT(M,N)$ that satisfies properties (1) and (2) is the edge from $(M,M+N)$ to $(M-1,M+N-K)$. Indeed, the other possible edge from $(0,0)$ to $(1,0)$ does not satisfy property (1):  the left edge of the triangle has larger slope than the line segment from $(1,0)$ to $(M-1,M+N-K)$.


Now consider the edge from $(M,M+N)$ to $(M-1,M+N-K)$. Lemma~\ref{lem-Newt} shows that  $f(x,y)$ has the form $f(x,y) = \pm 1 +b x^{M-1}y^{M+NK} +\ldots$, where $b\neq \pm 1$ when $N>K-2$; hence (3) is not satisfied. When $N=K-2$ then Lemma~\ref{lem-edge} shows that the edge with slope $K$ in the Newton polygon has length $K-3$, so (2) is not satisfied when $K>4$.  The only case remaining to consider is 
$K=4, M=3$ and $N=2$. The triangle that we get from this edge contains $9$ lattice points; hence (4) is not satisfied. Then, properties (1)-(4) determine the left edge of the triangle in all cases.

This finishes the proof of Theorem~\ref{thm-main-A}.

\subsection{Proof of Theorem~\ref{thm-rec}.}
It is easy to check that the three involutions map integral solutions of equation (\ref{eqn-MN}) to integral solutions and satisfy the given relation. It is also clear from the equation that there is no integral solution with $M<0, N>0$.  

For $K=3$ one can find that the only solution is $(M,N)=(1,1)$. For $K\geq 4$, the solutions lie on the two branches of a hyperbola, one branch with $M>N$ and the other with $M<N$. The three involutions exchange the branches.  In particular, starting from the branch where $M<N$ and applying $\iota_1$, we decrease $N$ and leave $M$ the same. Given a positive integer solution $(M,N)$, we apply $\iota_1$ if $M<N$ and we apply $\iota_0$ otherwise. A sequence of these involutions will keep $M,N$ non-negative, and it will decrease the sum $M+N$ until $(M,N)=(0,1)$. 

\section{Proof of Theorem~\ref{thm-main-B}}

Let $\Delta$ be a triangle as in Theorem~\ref{thm-main-B} and let $Z$ be the convex hull of its lattice points. Then $Z$ is an integral polytope with top edge from $(0,0)$ to $(m,h)$ and at least two edges on the lower boundary. We will assume that $\Delta$ is a minimal triangle containing $Z$, by which we mean that among the sequence of edges on the lower boundary of $Z$ the first and last lie on the edges on $\Delta$.

Our main tool will be Pick's formula applied to Z:
\[ C = A + \frac{P}{2} + 1,\]
where $P$ is the lattice perimeter of $Z$, $A$ is its area, and $C$ is the number of lattice points in $Z$.

Consider the case where $\Delta$ is either $IT(M,N)$ or $RT(M,N)$. In the integral case $Z$ is equal to $\Delta$ and its lower boundary consists of two edges with integral slopes $0$ and $K$. In the rational case $Z$ is a $4$-gon. Its lower boundary consists of three edges with integral slopes $0$, $L$, $K$, where $L=1$ or $L=K-1$, and the edge with slope $L$ has lattice length $1$. Let us prove that these properties of the lower edges of $Z$ characterize the triangles $IT(M,N)$ and $RT(M,N)$.

\begin{lemma} \label{lem-slopes}
Let $\Delta$ be a triangle as in Theorem~\ref{thm-main-B} and $Z$ the convex hull of its lattice points. Assume that $\Delta$ is the minimal triangle containing $Z$.
\begin{enumerate}
\item If $Z$ is a triangle with lower edges having integral slopes $0$ and $K$, then $\Delta$ is equal to $IT(M,N)$ for some $M,N$ satisfying equation (\ref{eqn-MN}).
\item If $Z$ has three lower edges with integral slopes $0$, $L$ and $K$, where $L=1$ or $L=K-1$ and the edge with slope $L$ has lattice length $1$, then  $\Delta$ is equal to $RT(M,N)$ for some $M,N$ satisfying equation (\ref{eqn-MN}).
\end{enumerate}
\end{lemma}

In both  cases of the lemma, the perimeter $P$ of $Z$ equals $m+1$. Let us first apply Pick's formula to this situation.

\begin{lemma}
If a lattice polygon $Z$ contains at least ${m+1\choose 2}+1$ lattice points, has area $\leq \frac{m^2}{2}$ and  perimeter $m+1$, then the area of $Z$ is $\frac{m^2-1}{2}$ and it contains ${m+1\choose 2}+1$ lattice points.
\end{lemma}

\begin{proof}
Let the area of $Z$ be $\frac{m^2-\varepsilon}{2}$ for some integer $\varepsilon \geq 0$. Pick's formula applied to $Z$ gives
\[ {m+1\choose 2}+1 \leq C  = \frac{m^2-\varepsilon+m+1}{2} + 1 = {m+1\choose 2}+1 + \frac{1-\varepsilon}{2},\]
which simplifies to $\varepsilon\leq 1$. Notice that if $\varepsilon=0$ then $C$ is not an integer, hence $\varepsilon=1$ is the only possibility.
\end{proof}

\begin{proof}[Proof of Lemma~\ref{lem-slopes}] 
In both cases equation (\ref{eqn-MN}) follows from the area of $Z$ being $\frac{m^2-1}{2}$.
Let us do the integral case only. Denote by $M$ and $N$ the lattice lengths of the two lower edges of $Z$, so that $m=M+N$. The height of $Z$ is $NK$, hence twice its area is $MNK$. This must be equal to $(M+N)^2-1$, giving equation (\ref{eqn-MN}). A similar argument applies in the rational case.
\end{proof}

Let us now return to a general triangle as in Theorem~\ref{thm-main-B}. Notice that $Z$ has perimeter $m+1$ if and only if it contains a lattice point on its lower boundary for every $x= 0,1,\ldots,m$. This is equivalent to all lower edges having integral slopes.

\begin{lemma}
Let $\Delta$ be a triangle as in Theorem~\ref{thm-main-B}. Then the perimeter $P$ of $Z$ is equal to $m+1$.
\end{lemma}

\begin{proof}
We know that the perimeter of $Z$ cannot be more than $m+1$ because its top boundary consists of one edge of lattice length $1$.

By assumption, the area of $Z$ is $\leq \frac{m^2}{2}$ and its number of lattice points is $\geq {m+1\choose 2}+1$. Pick's formula now gives that $P\geq m$. 

Let us rule out the case $P=m$. In that case, again from Pick's formula, the area of $Z$ is exactly $\frac{m^2}{2}$. Hence $Z$ is an integral triangle. One of its lower edges must have integral slope, the other edge has rational slope. Using a linear transformation, we may assume that $Z$ has vertices $(0,0)$, $(m,h)$, and $(m-2, 0)$. We get that twice the area of $Z$ is $m^2 = (m-2)h$. This gives a contradiction to $m$ and $h$ being relatively prime. 
\end{proof}

The previous lemma shows that the lower edges of $Z$ have integral slopes. We can apply a shear transformation $(x,y)\mapsto (x,y+ax)$ so that the first of the lower edges has slope $0$.  We need to show that these slopes are as in the case of $IT(M,N)$ or $RT(M,N)$.

\begin{lemma} \label{lem-Zslopes}
Let $\Delta$ be a triangle as in Theorem~\ref{thm-main-B}. Then, after a shear transformation, $Z$ is either a triangle with lower edges having integer slopes $0$ and $K$,  or $Z$ is a $4$-gon with lower edges having integer slopes $0, L, K$, where $L= 1$ or  $L= K-1$ and the edge with slope $L$ has lattice length $1$.
\end{lemma}

\begin{proof}
Assume that $Z$ is not a triangle, and  let its lower edges have integer slopes, with the first edge being horizontal.  Then $Z$ cannot have more than $3$ edges on its lower boundary. Otherwise the triangle $\Delta$ contains a lattice point on the $x$-axis that does not lie in $Z$. Thus, $Z$ must be a $4$-gon with lower edges having integer slopes $0, L, K$. 

We know that $Z$ has area $\frac{m^2-1}{2}$ and $\Delta$ has area $\leq \frac{m^2}{2}$. This implies that the edge with slope $L$ has lattice length $1$, and the area of the small triangle removed from $\Delta$ to obtain $Z$ is at most $\frac{1}{2}$. This small triangle has base $1-L/K$ and height $L$, giving the inequality
\[ \left(1-\frac{L}{K}\right)L \leq 1.\]
  The only integer solutions $K>L>0$ to this are $K>L=1$, $L=K-1>0$ and $L=2, K=4$. We need to rule out the last case. Here $Z$ has vertices $(0,0)$, $(a,0)$, $(a+1, 2)$ and $(m,h)=(a+1+b, 2+4b)$ for some integers $a,b>0$. Twice the area of $\Delta$ is
  \[ m^2 = \left(a+\frac{1}{2}\right) h.\]
This equality is equivalent to $a=b$. However, in that case $m=2a+1$ and $h=4a+2$ are not relatively prime. 
\end{proof}

Combining Lemma~\ref{lem-Zslopes} with Lemma~\ref{lem-slopes}, we get that the triangle $\Delta$ has to be isomorphic to either $IT(M,N)$ or $RT(M,N)$ for some $M,N$. This finishes the proof of Theorem~\ref{thm-main-B}.

\section{Mori Dream Spaces}

In this section we will consider triangles $\Delta$ that support a negative curve $C$. We start with general triangles and then specialize to the case $IT(M,N)$, $RT(M,N)$. 

\subsection{General triangles.}

Consider a triangle $\Delta$  of the form $(\dagger)$. We assume that $\Delta$ supports a negative curve $C^0$ that vanishes to order $m$ at $e$. If $C^0$ is defined by a polynomial $\xi$, we further assume that the lattice points $(0,0)$ and $(m,h)$ lie in the Newton polygon of $\xi$.  We make $\Delta$ larger by increasing its base:
\[ \Delta_{\alpha,\beta}: (-\alpha, 0), (b+\beta,0), (m,h), \quad \alpha,\beta \geq 0.\]
Let $X=\Bl_e X_{\Delta_{\alpha,\beta}}$. We assume that $\alpha,\beta\geq 0$ are small enough so that the strict transform $C$ of $C^0$ is a negative curve in $X$. This condition is equivalent to the area of $\Delta_{\alpha,\beta}$ being $\leq \frac{m^2}{2}$, which is the same as
\[ \alpha+\beta \leq \frac{m^2}{h} -b.\]
We study curve classes in $X$.
Let $D_0$ be in the class $H'-hE$ where $H'$ corresponds to the triangle with vertices:
\[ \Delta_{\alpha,\beta}': (0, 0), (m,0), \frac{m}{b+\alpha+\beta} (m+\alpha,h).\]

\begin{lemma}
$D_0 \cdot C = 0$.
\end{lemma}

\begin{proof}
Let the class of $C$ be $H-mE$. Here $H^2$ equals twice the area of $\Delta$, which is $bh$. When we scale the triangle $\Delta$ to $\Delta'$ with base of length $b'$, then the corresponding classes multiply as
\[ H\cdot H' = H \cdot \frac{b'}{b} H = b'h.\]
Now
\[ C\cdot D_0 = (H-mE)(H'-hE) = H\cdot H' - mh = 0.\]
\end{proof}

Cutkosky \cite{Cutkosky} has proved (see also \cite{GGK1}) that the variety $X=\Bl_e X_{\Delta_{\alpha,\beta}}$ is a MDS if and only if there exists a divisor $D$ in the class $\mu D_0$ for some integer $\mu > 0$ such that $C\cap D = \emptyset$. This intersection property is equivalent to $D$ not having $C$ as a component, which can be checked by finding a vertex of the triangle $\Delta_{\alpha,\beta}$ that does not lie in the Newton polygon of the polynomial defining $C$, but the corresponding vertex of $\Delta_{\alpha,\beta}'$ lies in the Newton polygon of the polynomial defining $D$. Then $C$ passes through the corresponding $T$-fixed point, but $D$ does not.

\begin{lemma}
When $\alpha=0, \beta\geq 0$ then $X=\Bl_e X_{\Delta_{0,\beta}}$ is a MDS.
\end{lemma}

\begin{proof} 
The polynomial $x^m (1-y)^h$ is supported in $\Delta'_{0,\beta}$ and vanishes to order $h$ at $e$. This defines the divisor $D$ disjoint from $C$.
\end{proof}

\begin{lemma} \label{lem-nonMDS}
Assume that the slope of the right edge of $\Delta$ is an integer $K$ such that 
\[ h\left(b+ \frac{1}{K}\right)> m^2.\]
Then $X=\Bl_e X_{\Delta_{\alpha,\beta}}$ is not a MDS for any $\alpha,\beta > 0$ such that the curve $C$ has non-positive self-intersection.
\end{lemma} 

\begin{proof}
Suppose we have a divisor $D$ in the class $\mu D_0$ that is disjoint from $C$.  Let $D$ be defined by a polynomial $\zeta$ supported in $\mu \Delta_{\alpha,\beta}'$. Then $\zeta$ is also supported in the larger triangle $\mu \Delta_{0,0}'$ where we have set $\alpha=\beta=0$. The triangle $\mu \Delta_{0,0}'$ supports another polynomial $x^{\mu m}(1-y)^{\mu h}$. We will use the two polynomials to get a contradiction.

Since $D\cdot C = 0$, the sheaf $\cO_X(D)$ restricts to the trivial sheaf on $C$. Hence any two global sections of the sheaf must be constant multiples of each other when restricted to $C$. Let $\xi$ be the polynomial defining $C$. Then
\[ \zeta \equiv c x^{\mu m}(1-y)^{\mu h} \mod \xi\]
for some constant $c\neq 0$. Let
\begin{equation} \label{eqn-zeta} \zeta  = c x^{\mu m}(1-y)^{\mu h} + \xi g , \end{equation}
where $g$ is a polynomial that vanishes at $e$ to order at least $\mu h - m$ and is supported in the triangle $\Delta''$ with base $\mu m - b$ and sides parallel to $\Delta_{0,0}$. However, since $D$ was assumed to exist when $\beta>0$, the right edge of the triangle $\Delta''$ should not intersect the Newton polygon of $g$. Thus, $g$ is in fact supported in a smaller triangle, still with sides parallel to $\Delta_{0,0}$. Since the right side of $\Delta_{0,0}$ has integer slope $K$, we can take the smaller triangle by shortening the base of $\Delta''$ by $1/K$. In summary, $g$ is supported in the triangle with base $[0,\mu m - b - 1/K]$ and must vanish to order at least $\mu h - m$ at $e$.  

Consider the terms in Equation (\ref{eqn-zeta}) lying on the left edge of $\Delta_{0,0}$. The Newton polygon of $\zeta$ intersects the left edge only at $(0,0)$  because $\alpha>0$. The Newton polygon of $x^{\mu m}(1-y)^{\mu h}$ intersect the left edge at $\mu(m,h)$ only. The Newton polygon of $\xi$ has two lattice points on the left edge, $(0,0)$ and $(m,h)$.  Equation (\ref{eqn-zeta}) now implies that  $g$ is not divisible by $\xi$. (By the same reason why a polynomial $1-ax^\mu$ is not divisible by $(1-bx)^2$ when $a,b\neq 0$ and the  characteristic is $0$.)

To get a contradiction to the existence of $D$, we now check that $\overline{D}\cdot C < 0$  in $\Bl_e X_{\Delta_{0,0}}$, where $\overline{D}$ is defined by the polynomial $g$. This implies that $\xi$ must divide $g$, which is impossible. We have 
\[ \overline{D}\cdot C \leq \left(\mu m - b - \frac{1}{K}\right)h - (\mu h - m) m = -\left(b+\frac{1}{K}\right)h+ m^2.\]
The last quantity is negative by assumption.
\end{proof}

\subsection{The triangles $IT(M,N)$ and $RT(M,N)$.}

We now specialize to the case of triangles described in Theorem~\ref{thm-main-A}.  Let $\xi^{int}_{M,N}$ and $\xi^{rat}_{M,N}$ be the polynomials defining negative curves in the two families of triangles, normalized to have constant term equal to $1$.

The triangles $IT(M,N)$ and $RT(M,N)$ are of the form $(\dagger)$. In particular, the varieties $X$ are MDS when $\alpha=0$. Let us prove the same for $\beta=0$.

\begin{lemma} Let $\Delta$ be either $IT(M,N)$ with $N>0$ or $RT(M,N)$ with $M>0$, and let $X$ be constructed by choosing $\alpha,\beta \geq 0$. If $\beta=0$ then $X$ is a MDS.
\end{lemma}

\begin{proof}
Consider the class $[D_0]$ defined in the previous subsection. In the integral case it has the form $H'-KN E$, where $H'$ corresponds to the triangle with base $[0,M+N]$. The polynomial 
\[ \big(\xi^{rat}_{N, (K-1)N-(M+N)} \big)^K = \big(\xi^{rat}_{\iota_1 \iota_0(M,N)} \big)^K\]
defines a divisor $D$ in the class $[D_0]$, disjoint from $C$.

In the rational case the class $[D_0]$ has the form $H'-(M+N)E$, where $H'$ corresponds to the triangle with base $M$. The polynomial  $\xi^{int}_{M,N}$ defines a divisor $D$ in the class $[D_0]$ that is disjoint from $C$.
\end{proof}

Let us now check when Lemma~\ref{lem-nonMDS} applies.

\begin{lemma}
Let $X$ be defined by choosing $\alpha,\beta>0$. Then $X$ is not a MDS when $N>1$ in the integral case and $M+N >1$ in the rational case.
\end{lemma}

\begin{proof}
In the integral case $m= M+N$, $h=NK$ and $b=M$. The inequality in Lemma~\ref{lem-nonMDS} becomes
\[ NK\left(M+\frac{1}{K}\right)>(M+N)^2 = MNK+1.\]
This is equivalent to $N>1$.

In the rational case $m=M$, $h=M+N$ and $b=M-\frac{M+N}{K}$. The inequality now is
\[ (M+N)\left(M- \frac{M+N-1}{K}\right) > M^2.\]
This is equivalent to $M+N>1$.
\end{proof} 

This finishes the proof of Theorem~\ref{thm-MDS}.

\section{Recurrence relations}

In this section we give explicit computations of the pairs $(M,N)$ satisfying equation (\ref{eqn-MN}).  We then show how to find the polynomials $\xi^{int}_{M,N}$ and $\xi^{rat}_{M,N}$ inductively. 

We will consider $K\geq 4$ and pairs $M>N$. The other pairs can be obtained by switching $M$ and $N$. Let $\tau$ be the transformation $\tau(M,N) = \iota_1 \iota_0(M,N) = (N,(K-1)N-(M+N))$. Recall from Theorem~\ref{thm-rec} that, for a fixed $K$, all $(M,N)$ can be obtained by starting with $(1,0)$ and applying $\iota_1, \iota_0$ alternately. Hence all solutions $M>N$  can be reduced to $(1,0)$ by applying $\tau$:
\[  (1,0) \stackrel{\tau}{\longmapsfrom} (K-2,1) \stackrel{\tau}{\longmapsfrom} ( (K-2)^2 -1, K-2) \stackrel{\tau}{\longmapsfrom} \ldots.\]

\subsection{Finding $(M,N)$.}
When $K=4$ then all solutions $M>N$ have the form $M=N+1$.  We fix $K>4$. Let 
\[ (M_0,N_0) = (1,0),  \quad (M_n, N_n) = \tau^{-n}(M_0,N_0), \quad n>0.\] 
This sequence of pairs can be written in terms of a sequence of numbers:
\[ (M_n, N_n) = (F_{n+1}, F_n),\]
Where $F_0=0, F_1=1$ and 
\[ F_{n+2}=(K-2)F_{n+1}-F_{n}, \quad n\geq 0.\]
We can solve for $F_n$ using eigenvalues. Let 
\[ \lambda_{\pm} = \frac{K-2 \pm \sqrt{(K-2)^2-4}}{2}.\]
Then, for $n\geq0$,
\[ F_n = \frac{\lambda_+^n-\lambda_-^n}{\lambda_+-\lambda_-}.\]
More concretely, for $n>0$,
\[
F_n=  \sum_{i=0}^{n-1} (-1)^i { 2n-1-i \choose i}  K^{n-1-i}. 
\]
From this, one can easily write polynomial expressions for the vertices of all the triangles $IT(M,N)$ and $RT(M,N)$.

Another way to find $M_n,N_n$ is to use the continued fractions expression of $\lambda_+$:
\[ \lambda_+ = (K-2) - \frac{1}{(K-2) - \frac{1}{(K-2) - \ldots}}.\]
Then $M_n/N_n$ is the truncation of this continued fraction. For example,
\[ \frac{M_2}{N_2} = (K-2) - \frac{1}{(K-2)}  = \frac{(K-2)^2-1}{K-2} .\]

\subsection{Finding $\xi^{int}_{M,N}$ and $\xi^{rat}_{M,N}$.}
\label{sec-rec}

 Recall that these polynomials are normalized to have constant term $1$. Let us denote by $\varepsilon^{int}_{M,N}$ and $\varepsilon^{rat}_{M,N}$ the coefficient of the highest degree term in $\xi^{int}_{M,N}$ and $\xi^{rat}_{M,N}$, respectively. We will see below that these coefficients are equal to $\pm 1$.

\begin{lemma} \label{lem-relations}
The polynomials $\xi^{int}_{M,N}$ and $\xi^{rat}_{M,N}$ satisfy the relations
\begin{align*}
 \xi^{int}_{M,N} &= \xi^{rat}_{M,N} \xi^{rat}_{\tau(M,N)} -\varepsilon^{rat}_{M,N} x^M(y-1)^{M+N},\\
 \big(\xi^{rat}_{\tau(M,N)}\big)^K  &=  \xi^{int}_{M,N} \xi^{int}_{\tau(M,N)} -\varepsilon^{int}_{M,N} x^{M+N}(y-1)^{KN}.
\end{align*}
The first equality holds when $M>0$, the second one when $N>0$.
\end{lemma}

\begin{proof}
Consider the rational triangle $\Delta = RT(M,N)$, with $M>0$. We let $\alpha=\beta=0$ and consider the class $[D_0] = H'-(M+N) E$, where $H'$ corresponds to the triangle with base $[0,M]$. We know two divisors in the class $[D_0]$, defined by $x^{M}(1-y)^{M+N}$ from the case $\alpha=0$ and $\xi^{int}_{M,N}$ from the case $\beta=0$. These two polynomials must be constant multiples of each other modulo $\xi^{rat}_{M,N}$. Write
\[ \xi^{int}_{M,N} = \xi^{rat}_{M,N} g - \varepsilon^{rat}_{M,N} x^{M}(y-1)^{M+N}\]
for some $g$ with constant term $1$, supported in $RT(N,(K-1)N-(M+N)) = RT(\tau(M,N))$ and vanishing to order at least $N$ at $e$. There is only one such polynomial, $g=\xi^{rat}_{\tau(M,N)}$. 

Now consider the integral triangle $\Delta = IT(M,N)$, with $N>0$. Again let $\alpha=\beta=0$ and consider the class $[D_0] = H'-KN E$, where $H'$ corresponds to the triangle with base $[0,M+N]$. There are two divisors in the class $[D_0]$, defined by polynomials $x^{M+N}(1-y)^{KN}$ and $\big(\xi^{rat}_{\tau(M,N)} \big)^K$.  Write 
\[ \big(\xi^{rat}_{\tau(M,N)}\big)^K =  \xi^{int}_{M,N} g -  \varepsilon^{int}_{M,N} x^{M+N}(y-1)^{KN}\]
where $g$ is a polynomial with constant term $1$, supported in $IT(\tau(M,N))$, and vanishing to order at least  $KN-(M+N)$. Again, $g$ has to be equal to $\xi^{int}_{\tau(M,N)}$. 
\end{proof}

One can use the two formulas to compute the polynomials $\xi_{M,N}$ inductively for all $M>N$.  Assume that we know the coefficients $\varepsilon_{M,N}$.  Start with 
 \[ \xi^{int}_{1,0} = 1-x, \qquad \xi^{rat}_{1,0} = 1-xy.\]
 Then use the second equation to solve for $\xi^{int}_{M,N}$ from $\xi^{int}_{\tau(M,N)}, \xi^{rat}_{\tau(M,N)}$, and the first equation to find $\xi^{rat}_{M,N}$ from $\xi^{int}_{M,N}, \xi^{rat}_{\tau(M,N)}$.
 
 One can also determine the coefficients $\varepsilon_{M,N}$ explicitly in the case $M>N$. The two equations give us the relations
\begin{align*}
  \varepsilon^{int}_{M,N} &= \varepsilon^{rat}_{M,N} \varepsilon^{rat}_{\tau(M,N)},\\
   \big(\varepsilon^{rat}_{\tau(M,N)}\big)^K &= \varepsilon^{int}_{M,N} \varepsilon^{int}_{\tau(M,N)}.
\end{align*}
Starting with 
 \[ \varepsilon^{int}_{1,0} = \varepsilon^{rat}_{1,0} = -1,\]
one can compute these coefficients inductively. Let $(M_n,N_n) = \tau^{-n}(1,0)$. Then for $K$ even, $\varepsilon^{int}_{M_n,N_n} = -1$ and $\varepsilon^{rat}_{M_n,N_n} = (-1)^{n+1}$. When $K$ is odd then 
\[ \varepsilon^{int}_{M_n,N_n} =  \begin{cases}1 & \text{if $n\equiv 1 \mod 3$,}\\ -1 & \text{otherwise.} \end{cases} \qquad
\varepsilon^{rat}_{M_n,N_n} =  \begin{cases}1 & \text{if $n\equiv 2 \mod 3$,}\\ -1 & \text{otherwise.} \end{cases} \]

\subsection{Newton polygons of $\xi^{int}_{M,N}$ and $\xi^{rat}_{M,N}$.}

When proving in Section~\ref{section.proofs.Thm1.1.A.Thm1.2} that the negative curves in Theorem~\ref{thm-main-A} are pairwise not isomorphic, we used some facts about the Newton polygons of their defining equations. Let us prove these facts. 

\begin{lemma}\label{lem-IT}
  The Newton polygon of $\xi_{M,N}^{int}$ is $IT(M,N)$.
\end{lemma}
\begin{proof}
  By the discussion above it's enough to show that the monomial $x^M$ appears in $\xi_{M,N}^{int}$ with a nonzero coefficient. The bases of the triangles $RT(M,N)$ and $RT(\tau(M,N))$ add up exactly to $M$. These two triangles don't have a monomial in their bottom right vertices, so the product of these two polynomials cannot contain the monomial $x^M$. Then, by Lemma \ref{lem-relations}, the only contribution to this monomial in $\xi_{M,N}^{int}$ comes from $x^M(1-y)^{M+N}$.
\end{proof}

The Newton polygon of $\xi^{rat}_{M,N}$ may be smaller than the convex hull of lattice points in $RT(M,N)$. For example, the Newton polygon of $\xi^{rat}_{8,3}$ does not contain the point $(5,0)$.  We know that the Newton polygons contain the left edge of the corresponding triangle. We want to show that the two adjacent edges lie on the other two edges of the triangle.  Let us write the polynomials as 
\begin{align*}
\xi_{M,N}^{rat} &= 1+ a_{M,N}^{rat} x + \ldots + b_{M,N}^{rat} x^{M-1} y^{M+N-K} \pm  x^{M} y^{M+N},\\
\xi_{M,N}^{int} &= 1+ a_{M,N}^{int} x + \ldots + b_{M,N}^{int} x^{M+N-1} y^{NK-K} \pm  x^{M+N} y^{NK}.
\end{align*}

\begin{lemma}\label{lem-Newt}
Let $(M_n,N_n) = \tau^{-n}(1,0)$.
\begin{enumerate}
\item The coefficients $a_{M_n,N_n}^{int}$, $b_{M_n,N_n}^{int}$ are nonzero when $n>0$. They are not equal to $\pm 1$ when $n>1$.
\item The coefficients $a_{M_n,N_n}^{rat}$, $b_{M_n,N_n}^{rat}$ are nonzero when $n>1$. They are not equal to $\pm 1$ when $n>2$.
\end{enumerate}
\end{lemma}

\begin{proof}
Let us start with the coefficients $a_{M,N}^{rat}$ and $a_{M,N}^{int}$.  The linear terms of the equations in Lemma~\ref{lem-relations} give the relations
\begin{align*}
 a_{M,N}^{int} &= a_{M,N}^{rat} + a_{\tau(M,N)}^{rat},\\
 K a_{\tau(M,N)}^{rat} &= a_{M,N}^{int} + a_{\tau(M,N)}^{int}.
\end{align*}
Starting with $a_{1,0}^{int} = -1$ and $a_{1,0}^{rat} = 0$, one can inductively find all these coefficients.

The two equations   above can be written in the form
\begin{align*}
a_{M,N}^{rat} &= (K-1) a_{\tau(M,N)}^{rat} - a_{\tau(M,N)}^{int}, \\
 a_{M,N}^{int} &= K a_{\tau(M,N)}^{rat} - a_{\tau(M,N)}^{int}.
\end{align*}
Let us take new variables $x_n = a_{M_n,N_n}^{rat}$, $y_n = a_{M_n,N_n}^{int} - a_{M_n,N_n}^{rat}$, where $(M_n, N_n) = \tau^{-n}(1,0)$. Then the equations have the form
\begin{align*}
x_{ n+1} &= (K-2) x_n - y_n,\\
 y_{n+1} &= x_n,
\end{align*}
with initial condition $x_0=0$, $y_0 = -1$. This is the same recurrence relation that we saw above when computing $(M,N)$. The solutions are $x_n = F_n$, $y_n = F_{n-1}$. Going back to the original variables,
\[ a_{M_n,N_n}^{rat} = F_n, \quad a_{M_n,N_n}^{int} = F_n+F_{n-1}.\]
From $F_n>0$ when $n>0$ and $F_n>1$ when $n>1$ we get the statements about $a_{M_n,N_n}^{rat}$ and $a_{M_n,N_n}^{int}$.

For the coefficients $b_{M,N}^{rat}$ and $b_{M,N}^{int}$ we can use the same argument as above. It is best to normalize the polynomials $\xi^{int}_{M,N}$ and $\xi^{rat}_{M,N}$ so that their highest degree terms have coefficient $1$. Then $b_{M,N}^{rat}, b_{M,N}^{int}$ satisfy the same recurrence relations as $a_{M,N}^{rat}, a_{M,N}^{int}$, but with different initial values. We need to make sure that the last terms in the equations of Lemma~\ref{lem-relations} do not affect the relations. This is true when $(M,N) = (M_n,N_n)$ with $n\geq 2$. Thus, given  $b_{M_1,N_1}^{rat}$ and $b_{M_1,N_1}^{int}$, we can compute all these numbers for $n\geq 2$. We have $b_{M_1,N_1}^{rat} = 0$ because the right edge of the triangle contains only one lattice point. To find $b_{M_1,N_1}^{int}$, we apply the first equation in Lemma~\ref{lem-relations} in the case $(M,N) = (M_1,N_1)$, taking into account the last term of the equation. Since we normalized $\xi^{int}_{M,N}$ so that its highest degree coefficient is $1$, we get \[ b_{M_1,N_1}^{int} = - \varepsilon_{M_1,N_1}^{int} \varepsilon_{M_1,N_1}^{rat} (-1)^{M_1+N_1} = - \varepsilon_{M_0,N_0}^{rat} (-1)^{K-1} = (-1)^{K-1}.\]
The same computation as above now gives $b_{M_1,N_1}^{rat} = 0, b_{M_1,N_1}^{int} = (-1)^{K-1}$,
\[  b_{M_n,N_n}^{rat} = (-1)^K F_{n-1},  \quad b_{M_n,N_n}^{int} = (-1)^K ( F_{n-1} + F_{n-2}), \quad n\geq 2. \]
The statements about $b_{M_n,N_n}^{rat}$ and $b_{M_n,N_n}^{int}$ now follow from the properties of $F_n$.
\end{proof}

The previous lemma shows that the Newton polygon of  $\xi_{M_n,N_n}^{rat}$ when $n>1$ contains the left edge of the corresponding triangle, and in addition edges of nonzero length along the other two sides of the triangle. For $n=2$ we can compute the lattice length of the edge with slope $K$:

\begin{lemma} \label{lem-edge}
The Newton polygon of $\xi_{M_2,N_2}^{rat}$ has an edge of lattice length $K-3$ with vertices $(M_2,M_2+N_2)$ and $(M_2-(K-3) ,1)$.
\end{lemma}

\begin{proof}
We compute $\xi_{M_2,N_2}^{rat}$ using the equations in Lemma~\ref{lem-relations}. Let us change scaling and coordinates so that the top vertex of each triangle corresponds to the constant term $1$ in all polynomials $\xi$, and the next lattice point on the edge with slope $K$ corresponds to monomial $z$. Then, working modulo all monomials that do not lie on right edge of the triangle, we find:
\[ \xi_{M_1,N_1}^{int}  \equiv \frac{(\xi_{M_0,N_0}^{rat})^K \pm x^{M_1+N_1}}{ \xi_{M_0,N_0}^{int}} \equiv \frac{1 \pm z} {1},\]
\[ \xi_{M_2,N_2}^{int}  \equiv \frac{(\xi_{M_1,N_1}^{rat})^K \pm x^{M_2+N_2}}{ \xi_{M_1,N_1}^{int}} \equiv \frac{1 \pm z^{K-1}} {1 \pm z} \equiv 1 \pm z \pm z^2 \pm \ldots \pm z^{K-2},\]
\[ \xi_{M_2,N_2}^{rat}  \equiv \frac{\xi_{M_2,N_2}^{int} \pm x^{M_2}}{\xi_{M_1,N_1}^{rat}} \equiv \frac{ \xi_{M_2,N_2}^{int} \pm z^{K-2}}{1} \equiv  1 \pm z \pm z^2 \pm \ldots \pm z^{K-3}.\]
Notice that in the numerator on the last line the term $z^{K-2}$ must cancel as it is not supported in the triangle.  
\end{proof}

\bibliographystyle{plain}
\bibliography{cox}

\end{document}